\documentclass[12pt]{article}
\usepackage[a4paper,margin=1in]{geometry}
\usepackage{amsmath,amssymb,amsthm}
\usepackage{xcolor}
\usepackage{enumitem}
\usepackage{tikz}
\usetikzlibrary{arrows.meta,positioning,calc}
\usepackage{microtype}
\usepackage{hyperref}
\hypersetup{colorlinks=true,linkcolor=black,citecolor=black,urlcolor=blue}
\usepackage{indentfirst}
\newtheorem{theorem}{Theorem}[section]
\newtheorem{problem}{Problem}[section]
\newtheorem{claim}{Claim}
\newtheorem{property}{Property}
\newtheorem{proposition}[theorem]{Proposition}

\title{Complexity of Arc-Decompositions involving Perfect Matchings and Cycle Factors}
\author{Hangning Liu$^{1}$, J\o rgen Bang-Jensen$^{1,3}$, Jin Yan$^{1}$, Jia Zhou$^{2}$\footnote{Corresponding author. E-mail address: jiazhou\_99@163.com.}\\[2mm]
\small $^{1}$School of Mathematics, Shandong University, Jinan 250100, China\\
\small $^{2}$School of Mathematics and Statistics, Ningxia University, Yinchuan 750021, China\\
\small $^{3}$Department of Mathematics and Computer Science, University of\\ \small Southern Denmark, Odense DK-5230, Denmark}
\date{}

\begin{document}
\maketitle

\begin{abstract}
For two digraph properties $P_1$ and $P_2$, a $(P_1,P_2)$-arc-decomposition of a digraph $D$ is a partition $A(D)=A_1\mathbin{\dot\cup}A_2$ such that the spanning subdigraphs $D[A_1]$ and $D[A_2]$ have properties $P_1$ and $P_2$, respectively. 
We prove that it is NP-complete to decide whether a digraph admits an arc-decomposition with properties $(P_1,P_2)$ where $(P_1,P_2)\in \{$(perfect matching,  having no odd cycle), (perfect matching, strong), (perfect matching, having an out-branching), (cycle factor, having no odd cycle)$\}$. All four problems remain NP-complete even when restricted to strongly connected digraphs. We also prove that every strongly connected digraph with minimum semi-degree at least three contains two vertex-disjoint cycles whose simultaneous arc deletion leaves the digraph strongly connected. These results settle some open problems posed by Bang-Jensen, Bessy, Gon\c{c}alves, and Picasarri-Arrieta [Theoret. Comput. Sci. 928 (2022), 167--182].
\end{abstract}

\vspace{1ex}
\vspace{1ex}


\section{Introduction}

Partition and decomposition problems form a classical and
extensively studied area of graph theory, where one asks
whether the vertices, edges, or arcs of a graph can be divided into parts satisfying prescribed properties. Bern\'ath and Kir\'aly~\cite{BernathKiraly2015} studied the complexity of natural packing, covering and partitioning problems in undirected graphs. For digraphs, vertex-partition problems and their degree-constrained variants have been studied in~\cite{BangJensenBessyHavetYeo2018,BangJensenChristiansen2018,BangJensenCohenHavet2016,BangJensenHavet2016}.

The topic of this paper is the natural analogue where we seek to partition the arc set with prescribed properties. For two digraph properties $P_1$ and $P_2$, a {\bf $\mathbf{(P_1,P_2)}$-arc-decomposition} of a digraph $D$ is a partition $A(D)=A_1\mathbin{\dot\cup}A_2$ such that the spanning subdigraphs $D[A_1]$ and $D[A_2]$ have properties $P_1$ and $P_2$, respectively.

A natural example is the $(\mathrm{strong}, \mathrm{strong})$-arc-decomposition problem, usually referred to as the strong arc-decomposition problem. This problem is NP-complete already for 2-regular digraphs~\cite{BangJensenYeo2012}. 
On the other hand, Bang-Jensen and Yeo~\cite{BangJensenYeo2004} proved that every $2$-arc-strong semicomplete digraph admits a strong arc-decomposition, apart from one exceptional digraph. Wood~\cite{Wood2004} studied arc-decompositions into acyclic subdigraphs subject to out-degree bounds.

Bang-Jensen, Bessy, Gon\c{c}alves, and Picasarri-Arrieta~\cite{BJBGPA2022} systematically studied arc-decomposition problems involving various prescribed properties and established a broad classification of their computational complexity. Ai, Gutin, Lei and Shi \cite{aiarXiv2608.00115} proved that it is NP-complete to determine whether a digraph has a (strong, antistrong)-arc-partition and also whether it has an (antistrong, $2$-edge-connected)-arc-partition. H\"o{}rsch and Picasarri-Arrieta \cite{hoerschTEJC31} studied arc-decompositions into directed linear forests. 

A \textbf{matching} in a digraph is a set of arcs no two of which have a common end-vertex, and it is \textbf{perfect} if every vertex is incident with exactly one arc of the matching. A \textbf{cycle factor} is a spanning subdigraph in which every vertex has in-degree and out-degree one. Recently, Borsik and Madarasi~\cite{BorsikMadarasi2025} studied several $(\mathcal F,\mathrm{acyclic})$-arc-decomposition problems, including cases where $\mathcal F$ is a matching or a perfect matching. In particular, they proved that the (matching, acyclic)-arc-decomposition and the (perfect matching, acyclic)-arc-decomposition problems are NP-complete. Throughout this paper, the properties ``perfect matching'' and ``cycle factor'' require the corresponding spanning subdigraph to be a perfect matching or a cycle factor, respectively. In contrast, ``having a perfect matching'' and ``having a cycle factor'' require only that it contain one.

A digraph is \textbf{bipartite} if its underlying graph is bipartite. A digraph $T$ is an \textbf{out-tree} if $T$ is an oriented tree with just one vertex $s$ of in-degree zero. An \textbf{out-branching $(B^+)$} is a spanning out-tree. In 2022, Bang-Jensen et al.~\cite{BJBGPA2022} pointed out that one may consider $(P_1,P_2)$-arc-decomposition problems with other natural properties, including perfect matchings, cycle factors and bipartite subgraphs, and posed the following problem.

\begin{problem}[Bang-Jensen, Bessy, Gon\c{c}alves, and Picasarri-Arrieta~\cite{BJBGPA2022}]

What is the complexity of the following arc-decomposition problems?

\begin{itemize}

    \item the $(\text{cycle factor},\text{ having no odd cycle})$-arc-decomposition problem;

    \item the $(\text{perfect matching},\text{ having no odd cycle})$-arc-decomposition problem;

    \item the $(\text{perfect matching},\mathrm{strong})$-arc-decomposition problem;

    \item the $(\text{perfect matching},\text{ having }B^+)$-arc-decomposition problem;

    \item the $(\mathrm{bipartite},\Delta^+\leq k)$-arc-decomposition problem.

\end{itemize}

\end{problem}

The fifth problem is closely related to the vertex-partition problems studied by Bang-Jensen, Bessy, Havet, and Yeo~\cite{BangJensenBessyHavetYeo2018}. Indeed, a digraph $D$ admits a $(\mathrm{bipartite},\Delta^+\leq k)$-arc-decomposition if and only if $V(D)$ has a partition $(V_1,V_2)$ such that $\Delta^+(D[V_i])\leq k$ for $i=1,2$. Hence, for every fixed $k\geq1$, this problem is NP-complete by the results in~\cite{BangJensenBessyHavetYeo2018}.

The purpose of this paper is to settle the remaining four problems. Our main results are the following.

\begin{theorem}\label{thm:11}
It is NP-complete to decide whether a digraph has a $($perfect matching, having no odd cycle$)$-arc-decomposition.
\end{theorem}

\begin{theorem}\label{thm:12}
It is NP-complete to decide whether a digraph has a $(\text{perfect matching}$, $strong)$-arc-decomposition, even when restricted to strongly connected digraphs.
\end{theorem}

\begin{theorem}\label{thm:13}
It is NP-complete to decide whether a digraph has a $(\text{perfect matching},$ $ \text{having}$ $B^+)$-arc-decomposition.
\end{theorem}

\begin{theorem}\label{thm:14}

It is NP-complete to decide whether a digraph has a
$($cycle factor, having no odd cycle$)$-arc-decomposition.
\end{theorem}

A property $P$ is \textbf{upwards closed} (\textbf{downwards closed}) if every superdigraph (subdigraph) on the same vertex set of a digraph $D$ with property $P$ also has property $P$. The properties ``having a perfect matching'' and ``having a cycle factor'' are upwards closed, whereas ``having no odd cycle'' is downwards closed. The (having a perfect matching, having no odd cycle) and (having a cycle factor, having no odd cycle)-arc-decomposition problems are both polynomially solvable: it suffices to test whether the input digraph contains a perfect matching or a cycle factor, respectively, since the second part may be empty. In contrast, Theorems~\ref{thm:11} and~\ref{thm:14} show that the (perfect matching, having no odd cycle) and (cycle factor, having no odd cycle)-arc-decomposition problems are NP-complete. 

Since ``strong'' and ``having $B^+$'' are upwards closed, replacing ``perfect matching'' by ``having a perfect matching'' does not change the set of yes-instances in either problem. Thus, Theorems~\ref{thm:12} and~\ref{thm:13} also establish NP-completeness for the (having a perfect matching, strong) and (having a perfect matching, having $B^+$)-arc-decomposition problems, respectively.

Holyer's NP-completeness result for $3$-edge-colouring cubic
graphs \cite{Holyer1981} implies that deciding whether a graph contains two edge-disjoint perfect matchings is NP-complete. Orienting each edge arbitrarily yields the corresponding NP-completeness result for two arc-disjoint perfect matchings in digraphs. We show that this problem and the problems in Theorems~\ref{thm:11}-\ref{thm:14} remain NP-complete for strong digraphs. Apart from the case already covered by Theorem~\ref{thm:12}, these results are collected in the following theorem.

\begin{theorem}\label{thm:15}
The following arc-decomposition problems are all NP-complete even when restricted to strong digraphs.

\begin{itemize}
    \item[$\big($i$\big)$] The $\big($perfect matching, having no odd cycle$\big)$-arc-decomposition problem.
    \item[$\big($ii$\big)$] The $\big($perfect matching, having $B^+$$\big)$-arc-decomposition problem. 
    \item[$\big($iii$\big)$] The $\big($cycle factor, having no odd cycle$\big)$-arc-decomposition problem.
     \item[$\big($iv$\big)$] The $\big($having a perfect matching, having a perfect matching$\big)$-arc-decomposition problem.
\end{itemize}
\end{theorem}

We also consider the deletion of vertex-disjoint cycles
while preserving strong connectivity. Bang-Jensen et al.~\cite{BJBGPA2022} proved the following.

\begin{proposition}[{\cite[Proposition~15]{BJBGPA2022}}]\label{thm:cycleandstrong}
Every strong digraph $D=(V,A)$ with minimum out-degree at least $2$ has a cycle $C$ such that $D-A(C)$ is strong.
\end{proposition}

They asked whether stronger degree or arc-connectivity conditions guarantee two vertex-disjoint cycles whose simultaneous arc deletion leaves the digraph strong. We show that minimum semi-degree at least $3$ suffices.

\begin{theorem}\label{thm:16}
Every strong digraph $D$ with minimum semi-degree at least $3$ contains two vertex-disjoint cycles $C_1,C_2$ such that $D - A(C_1) - A(C_2)$ is strongly connected.
\end{theorem}

The minimum semi-degree bound of $3$ is best possible, since the complete symmetric digraph on three vertices has minimum semi-degree $2$ but contains no two vertex-disjoint cycles.

Theorem~\ref{thm:11} is proved by a reduction from $3$-SAT. The proofs of Theorems~\ref{thm:12} and~\ref{thm:13} use reductions from the Hamiltonian cycle problem for $2$-regular digraphs, which is known to be NP-complete~\cite{Ramanath1985}. The proof of Theorem~\ref{thm:14} uses the construction in~\cite[Section~4.3.3]{BJBGPA2022}, with a modified verification that the remaining subdigraph contains no odd cycles. The full proof is given in the Appendix. 
The proofs of parts (i)--(iv) of Theorem~\ref{thm:15} use reductions from the unrestricted versions of the corresponding arc-decomposition problems.

The rest of the paper is organized as follows. In Section~2, we recall the notation and preliminary results. In Section~3, we prove Theorem~\ref{thm:11}.  The proofs of Theorems~\ref{thm:12} and~\ref{thm:13} are presented in Section~4. Section~5 is devoted to the proof of Theorem~\ref{thm:15}. Section~6 is devoted to the proof of Theorem~\ref{thm:16}. 

\section{Preliminaries}

Notation not introduced here is consistent with~\cite{BangJensenGutin2009}. For an integer $i$, we use the notation $\boldsymbol{[i]=\{1,\ldots,i\}}$. The digraphs considered in this paper are always simple, i.e., without loops and multiple arcs. Pairs of oppositely directed arcs are allowed. The order and size of a digraph $D$ are denoted by $\boldsymbol{|D|}$ and $\boldsymbol{\|D\|}$, respectively. For two vertices $x,y\in V(D)$, we denote the arc from $x$ to $y$ by $xy$. For $A'\subseteq A(D)$, we write $\boldsymbol{D[A']=(V(D),A')}$ for the spanning subdigraph of $D$ with arc set $A'$. For $A'\subseteq A(D)$, we write $\boldsymbol{D-A'}$ for the digraph obtained from $D$ by deleting all arcs in $A'$.

For a vertex $x$ of $D$, its out-degree and in-degree are denoted by $\boldsymbol{d_D^+(x)}$ and $\boldsymbol{d_D^-(x)}$, respectively, and $\boldsymbol{d_D(x)=d_D^-(x)+d_D^+(x)}$. The maximum out-degree of $D$ is denoted by $\Delta^+(D)$. The minimum out-degree and minimum in-degree of a digraph are denoted by $\boldsymbol{\delta^+(D)}$ and $\boldsymbol{\delta^-(D)}$, respectively. The minimum semi-degree of $D$ is defined as $\boldsymbol{\delta^0(D)=\min\{\delta^+(D),\delta^-(D)\}}$. The underlying graph $\boldsymbol{U(D)}$ of a digraph $D$ is obtained by replacing each arc of $D$ by an undirected edge and suppressing multiple edges. A digraph is \textbf{strong} if it contains an $(x,y)$-path for every ordered pair of distinct vertices $x,y$.

We shall repeatedly use the \textbf{vertex-splitting} construction. Given a digraph $D$, let $\boldsymbol{S(D)}$ be the digraph with
\[
{V(S(D))=\{v^-,v^+:v\in V(D)\}}
\]
and
\[
{A(S(D))
=\{v^-v^+:v\in V(D)\}\cup\{v^+u^-:vu\in A(D)\}.}
\]
Thus each vertex $v$ is replaced by the arc $v^-v^+$, and each arc $vu$ is replaced by $v^+u^-$. This is the standard {\bf vertex-splitting} procedure; see~\cite{BangJensenGutin2009}.

We shall use the following known result. Recall that a digraph is \textbf{$\boldsymbol{k}$-regular} if every vertex has in-degree and out-degree $k$.

\begin{theorem}[{\cite[Theorem~6.1.2]{BangJensenGutin2009}}]\label{thm:Ramanath}
It is NP-complete to decide whether a $2$-regular digraph contains a Hamiltonian cycle.
\end{theorem}

\begin{theorem}[\cite{Thomassen1983}]\label{thm:2cycles}
Every digraph of minimum out-degree at least $3$ contains $2$ pairwise vertex-disjoint cycles.
\end{theorem}

\section{Proof of Theorem~\ref{thm:11}}

\begin{proof}

We prove the theorem by a reduction from 3-SAT. Given an instance $F$, we construct a digraph $D_F$ using variable and clause gadgets. The variable gadgets encode truth assignments through perfect matchings, while the clause gadgets enforce clause satisfaction through the absence of odd cycles. We show that $F$ is satisfiable if and only if $D_F$ admits a (perfect matching, having no odd cycle)-arc-decomposition.

Let $F$ be a 3-SAT instance with variables $x_1,\ldots,x_n$ and clauses $C_1,\ldots,C_m$. Note that we require every clause to consist of three pairwise distinct and non‑complementary literals, and every variable to occur in at least one clause. The prescribed order of clauses naturally defines an order for all occurrences of each variable $x$ and its negation $\bar x$ within $F$. For nonnegative integers $p,q$ with $p+q>0$, define the variable gadget $W[u,v;p,q]$ with vertex set
\[
V(W)=\{u,v,y_1,\ldots,y_{2p},z_1,\ldots,z_{2q}\}
\]
and arc set
\[
A(W)=A(P_y)\cup A(P_z),
\]
where $P_y=uy_1\cdots y_{2p}v$ and $P_z=uz_1\cdots z_{2q}v$. If $p=0$ (respectively, $q=0$), the corresponding path is understood to be the single arc $uv$. The variable gadget $W[u,v;3,3]$ is illustrated in Figure~\ref{fig:1}.

\begin{figure}[h]
    \centering
    \includegraphics[width=0.4\hsize]{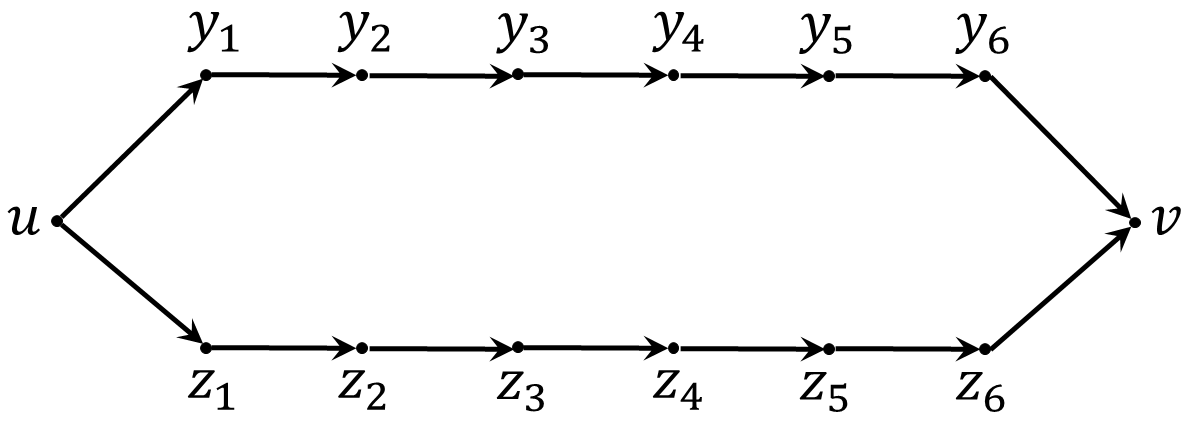}
    \caption{$W[u,v;3,3]$}
    \label{fig:1}
\end{figure}

 We next construct the clause gadget $R[r_1,r_2,s_1,s_2,t_1,t_2;e_1,e_2,e_3]$, where $r_1,r_2,s_1,s_2,t_1,t_2$ are vertices and $e_1=a_1b_1$, $e_2=a_2b_2$, $e_3=a_3b_3$ are arcs, disjoint from these vertices (see Figure \ref{fig:2}). The gadget $R$ has 12 vertices and the arc set consists of a directed 9-cycle $a_1b_1r_2a_2b_2s_2a_3b_3t_2a_1$ and three arcs $r_1r_2,s_1s_2,t_1t_2$. 

\begin{figure}[h]
    \centering
    \includegraphics[width=0.26\hsize]{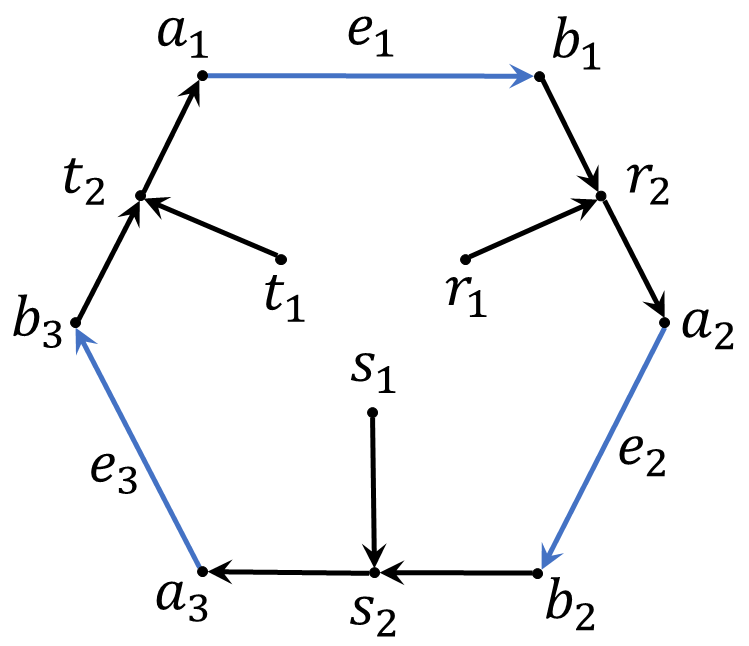}
    \caption{$R[r_1,r_2,s_1,s_2,t_1,t_2;e_1,e_2,e_3]$}
    \label{fig:2}
\end{figure}

We construct the digraph $D_F$ as follows:
\begin{enumerate}[label=(\arabic*)]
\item For each variable $x_i$, we assign a distinct copy $W[u_i,v_i;p_i,q_i]$ of the variable gadget, where $p_i$ (respectively, $q_i$) denotes the number of clauses in which $x_i$ (respectively, $\bar x_i$) occurs as a literal.
\item For each clause $C_j$, we assign one copy $R[r_{j,1},r_{j,2},s_{j,1},s_{j,2},t_{j,1},t_{j,2};e_{j,1},e_{j,2},e_{j,3}]$ of the clause gadget, where the six vertices $r_{j,1},r_{j,2},s_{j,1},s_{j,2},t_{j,1},t_{j,2}$ are private to this copy and the arcs $e_{j,1},e_{j,2},e_{j,3}$ are identified with arcs from the variable gadgets as follows. Suppose $C_j$ consists of literals from variables $x_i,x_k,x_\ell$, where each literal may be either a variable or its negation. If $x_i$ appears unnegated in $C_j$ and this is the $r$-th occurrence of $x_i$, then we identify $e_{j,1}$ with the arc $y_{i,2r-1}y_{i,2r}$. If $\bar x_i$ appears in $C_j$ and this is the $s$-th occurrence of $\bar x_i$, then we identify $e_{j,1}$ with the arc $z_{i,2s-1}z_{i,2s}$. The arcs $e_{j,2}$ and $e_{j,3}$ are identified analogously with arcs of variable gadgets.
\end{enumerate}

\begin{property}
Suppose $(A_1,A_2)$ is an arc-decomposition of $D_F$ such that $A_1$ forms a perfect matching and $A_2$ contains no odd cycle. Then the following properties hold:
\begin{enumerate}[label=\arabic*.]
\item The arcs $r_{j,1}r_{j,2},s_{j,1}s_{j,2},t_{j,1}t_{j,2}$ are in $A_1$ for every $j\in [m]$;
\item The arcs $b_{j,1}r_{j,2},r_{j,2}a_{j,2},b_{j,2}s_{j,2},s_{j,2}a_{j,3},b_{j,3}t_{j,2},t_{j,2}a_{j,1}$ all belong to $A_2$ for every $j\in [m]$;
\item Inside every variable gadget, the arcs of $A_1$ are formed by one of the two perfect matchings of the even cycle $U(W)$ formed by the two $(u,v)$-paths. Equivalently, when $p,q>0$, either
\[
uy_1,y_2y_3,\ldots,y_{2p-2}y_{2p-1},y_{2p}v,z_1z_2,\ldots,z_{2q-1}z_{2q}\in A_1,
\]
or
\[
uz_1,z_2z_3,\ldots,z_{2q-2}z_{2q-1},z_{2q}v,y_1y_2,\ldots,y_{2p-1}y_{2p}\in A_1,
\]
and all other arcs of the variable gadget belong to $A_2$. The same description is read in the natural way when $p=0$ or $q=0$.
\end{enumerate}
\end{property}

\begin{proof}
As $r_{j,1,}s_{j,1},t_{j,1}$ are of degree one and $A_1$ is a perfect matching, the arcs $r_{j,1}r_{j,2},s_{j,1}s_{j,2},t_{j,1}t_{j,2}$ belong to $A_1$. Therefore, the arcs $b_{j,1}r_{j,2},r_{j,2}a_{j,2},b_{j,2}s_{j,2},s_{j,2}a_{j,3},b_{j,3}t_{j,2},$ $t_{j,2}a_{j,1}$ all belong to $A_2$. Notice that every vertex inside each variable gadget $W$ has all its incident arcs going outside the gadget contained in $A_2$, because each such outside arc is incident with a copy of one of $r_2,s_2,t_2$, which is already matched in $A_1$ to the corresponding copy of $r_1,s_1,t_1$, respectively. Hence, the arcs of $A_1\cap A(W)$ form exactly a perfect matching of $V(W)$. Since the underlying graph of every variable gadget is an even cycle, it has exactly the two perfect matchings listed above. This proves the property.
\end{proof}

\begin{claim}
$F$ is satisfiable if and only if $D_F$ admits a (perfect matching, having no odd cycle)-arc-decomposition.
\end{claim}

\begin{proof}
Suppose $F$ is satisfiable, and let $\Phi$ be a satisfying truth assignment for $F$. For each variable $x_i$, if $x_i$ evaluates to true under $\Phi$, we include in $A_1$ the perfect matching of the underlying variable cycle containing the first edge of $P_z$; when $p_i,q_i>0$, these are the arcs $u_i z_{i,1}$, $z_{i,2}z_{i,3}$, $\ldots$, $z_{i,2q_i-2}z_{i,2q_i-1}$, $z_{i,2q_i}v_i$, $y_{i,1}y_{i,2}$, $\ldots$, $y_{i,2p_i-1}y_{i,2p_i}$. If $x_i$ evaluates to false under $\Phi$, we instead place into $A_1$ the other perfect matching, i.e., the one containing the first edge of $P_y$; when $p_i,q_i>0$, these are the arcs $u_i y_{i,1}$, $y_{i,2}y_{i,3}$, $\ldots$, $y_{i,2p_i-2}y_{i,2p_i-1}$, $y_{i,2p_i}v_i$, $z_{i,1}z_{i,2}$, $\ldots$, $z_{i,2q_i-1}z_{i,2q_i}$. We further add to $A_1$ the arcs $r_{j,1}r_{j,2}$, $s_{j,1}s_{j,2}$, $t_{j,1}t_{j,2}$ for every $j\in[m]$, and define $A_2$ to be the set of all remaining arcs. One can verify that every arc associated with a true literal lies in $A_1$.

Clearly, each vertex of $D_F$ is incident with exactly one arc within $A_1$. Consequently, $A_1$ forms a perfect matching.

We now show that $D_F[A_2]$ contains no odd cycle. In a variable gadget, each $A_2$-arc which is not a distinguished clause arc goes from a vertex of in-degree zero in $D_F[A_2]$ to a vertex of out-degree zero in $D_F[A_2]$, and hence cannot lie on a cycle. Thus every cycle of $D_F[A_2]$ must be contained in one clause gadget and, in fact, must be its directed 9-cycle. Since every clause has at least one true literal, at least one of $e_{j,1},e_{j,2},e_{j,3}$ lies in $A_1$, so that directed 9-cycle is broken. Therefore $D_F[A_2]$ has no odd cycle.

Conversely, suppose $D_F$ admits an arc-decomposition $(A_1,A_2)$ such that $A_1$ forms a perfect matching and $A_2$ contains no odd cycle. We define a truth assignment $\Phi$ as follows: set $x_i$ to true if the unique edge of $A_1$ incident with $u_i$ lies on $P_z$, and set $x_i$ to false if it lies on $P_y$. Property~1 shows that exactly one alternative occurs.

Consider an arbitrary clause $C_j$ of $F$. As noted earlier, all arcs of the 9-cycle in the corresponding clause gadget, except $e_{j,1},e_{j,2},e_{j,3}$, lie in $A_2$. Since the subdigraph induced by $A_2$ contains no odd cycle, at least one arc among $e_{j,1},e_{j,2},e_{j,3}$ must belong to $A_1$. Without loss of generality, assume that $e_{j,1}\in A_1$.

If the corresponding variable $x_i$ appears as a positive literal in $C_j$, then $e_{j,1}$ belongs to the path $P_y=u_i y_{i,1}\cdots y_{i,2p_i}v_i$. By Property~1, this forces the $A_1$-edge incident with $u_i$ to lie on $P_z$, so $x_i$ is true under $\Phi$. Likewise, if $x_i$ appears as a negative literal in $C_j$, then $e_{j,1}$ lies in the path $P_z=u_i z_{i,1}\cdots z_{i,2q_i}v_i$. It follows that the $A_1$-edge incident with $u_i$ lies on $P_y$, so $x_i$ is false under $\Phi$. Hence, $C_j$ contains a true literal. Since $C_j$ is chosen arbitrarily, every clause of $F$ is satisfied by the assignment $\Phi$.
\end{proof}

The construction takes polynomial time in $|F|$. The problem belongs to NP, since one can check in polynomial time that $A_1$ is a perfect matching and that the underlying graph of each strong component of $D_F[A_2]$ is bipartite, the latter being equivalent to $D_F[A_2]$ containing no odd cycle. Therefore, the problem is NP-complete. 
\end{proof}

\section{Proofs of Theorems~\ref{thm:12} and~\ref{thm:13}}
\subsection{(perfect matching, strong)-arc-decompositions}

\noindent\textit{Proof of Theorem~\ref{thm:12}.}
To show this, we reduce the Hamiltonian cycle problem for 2-regular digraphs to the (perfect matching, strong)-arc-decomposition problem. Let $D$ be a 2-regular digraph of order $n$ and let $D'$ be obtained from $D$ by the vertex-splitting procedure described in Section~2, that is, $V(D')=V'\cup V''$ and
\[
A(D')=\{v''u'\mid vu\in A(D)\}\cup\{v'v''\mid v\in V(D)\},
\]
where $V'$ and $V''$ are two copies of $V(D)$. The vertex-splitting construction $D'$ is illustrated in Figure~\ref{fig:3}. As $D$ is 2-regular, we have $|D'|=2n$, $\|D'\|=3n$, and, for all $x\in V(D')$, we have $d(x)=3$.

\begin{figure}[h]
    \centering
    \includegraphics[width=0.32\hsize]{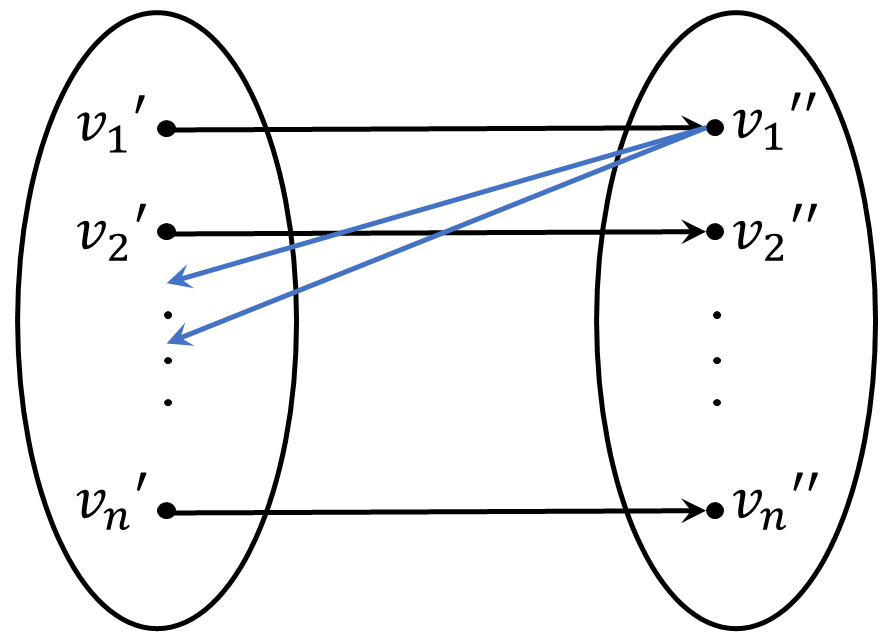}
    \caption{The digraph $D'$ obtained by splitting each vertex $v$ into $v'$ and $v''$.}
    \label{fig:3}
\end{figure}

\begin{claim}
$D$ has a Hamiltonian cycle if and only if $D'$ admits a $(\text{perfect matching}, \text{strong})$-arc-decomposition.
\end{claim}

\begin{proof}
Let $C=v_1v_2\cdots v_nv_1$ be a Hamiltonian cycle in $D$. By the construction of $D'$, it is straightforward to verify that the sequence
$C'=v'_1v''_1v'_2v''_2\cdots v'_nv''_nv'_1$ forms a Hamiltonian cycle in $D'$. Define $A_2=A(C')$ and $A_1=A(D')\setminus A_2$, and let $D_1=(V(D'),A_1)$ and $D_2=(V(D'),A_2)$. The digraph $D_2$ is strong, and every vertex $x\in V(D')$ satisfies $d_{D_2}(x)=2$ and $d_{D_1}(x)=1$. Therefore, $(A_1,A_2)$ yields an arc-decomposition of $D'$ into a perfect matching and a strong subdigraph.

Conversely, suppose $D'$ admits the desired arc-decomposition $(A_1,A_2)$. Clearly, $|A_1|=n$ and $|A_2|=3n-|A_1|=2n$. Since $D_2=(V(D'),A_2)$ is strong, every vertex has at least one in-arc and at least one out-arc in $D_2$. As $|A_2|=|V(D')|=2n$, it follows that every vertex has in-degree and out-degree exactly one in $D_2$. Hence $D_2$ is a cycle factor; since it is strong, it consists of a single cycle. Thus $A_2$ is the arc set of a Hamiltonian cycle of $D'$. By the bipartite form of the vertex-splitting construction, this Hamiltonian cycle alternates between $V'$ and $V''$, and the unique arc leaving each $v'\in V'$ is $v'v''$. Hence, after renaming the original vertices of $D$, we may write it as
$C'=v'_1v''_1v'_2v''_2\cdots v'_nv''_nv'_1$. It follows that $D$ contains the Hamiltonian cycle $C=v_1v_2\cdots v_nv_1$.
\end{proof}

As we can construct $D'$ from $D$ in linear time, the problem is NP-hard by Theorem~2.1. It also belongs to NP, since a proposed arc-decomposition can be checked in polynomial time: one verifies that $A_1$ is a perfect matching and that $(V(D'),A_2)$ is strong. Hence the (perfect matching, strong)-arc-decomposition problem is NP-complete. 
To extend the result to strong digraphs, we use the following reduction. For an instance $D$ of the original problem, set $D'=D$ if $D$ is strong; otherwise set $D'=C_3$. Clearly, $D$ admits a (perfect matching, strong)-arc-decomposition if and only if $D'$ admits such an arc‑decomposition. Therefore, the (perfect matching, strong)-arc-decomposition problem is NP-complete even when restricted to strong digraphs.
\hfill$\square$

\subsection{(perfect matching, having $B^+$)-arc-decompositions}

\noindent\textit{Proof of Theorem~\ref{thm:13}.}
We show how to reduce the Hamiltonian cycle problem for 2-regular digraphs to the (perfect matching, having $B^+$)-arc-decomposition problem.

Let $D$ be a 2-regular digraph of order $n$, and fix an arbitrary vertex $w\in V(D)$. Let $D''$ be obtained from $D$ by the vertex-splitting procedure and then let $D'=D''-\{w'w''\}$. The construction of $D'$ is illustrated in Figure~\ref{fig:4}, where the dashed arc represents the deleted arc. 

It follows that $|V(D')|=2n$ and $|A(D')|=3n-1$. For every vertex $v'\in V'\setminus\{w'\}$, we have $d^+_{D'}(v')=1$ and $d^-_{D'}(v')=2$. For every vertex $v''\in V''\setminus\{w''\}$, we have $d^+_{D'}(v'')=2$ and $d^-_{D'}(v'')=1$. In addition, $d^+_{D'}(w')=d^-_{D'}(w'')=0$ and $d^-_{D'}(w')=d^+_{D'}(w'')=2$.

\begin{figure}[h]
    \centering
    \includegraphics[width=0.32\hsize]{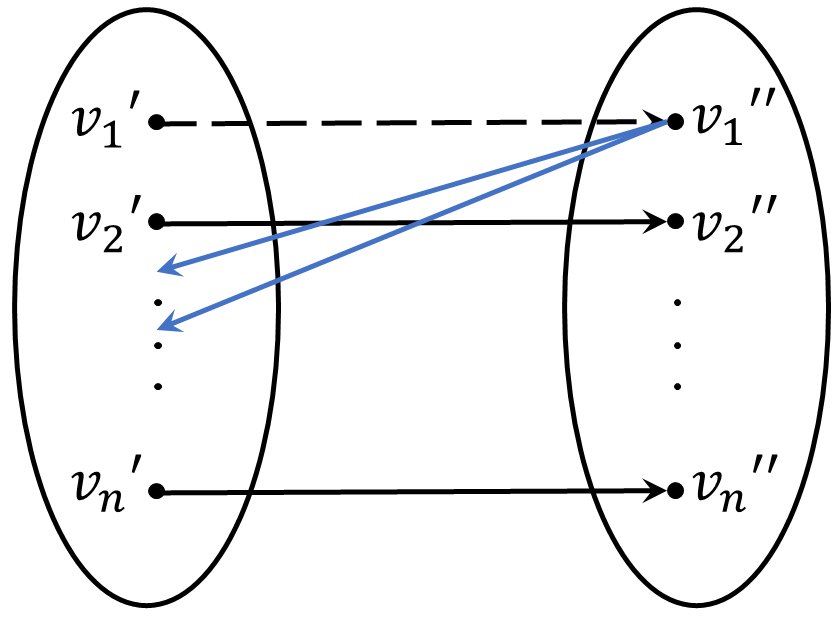}
    \caption{The digraph $D'$ obtained by deleting $w'w''$ from the vertex-split digraph, where $w=v_1$.}
    \label{fig:4}
\end{figure}

\begin{claim}
The digraph $D$ contains a Hamiltonian cycle if and only if $D'$ admits a $(\text{perfect matching}, \text{having}$ $B^+)$-arc-decomposition.
\end{claim}

\begin{proof}
Suppose $C=v_1v_2\cdots v_nv_1$ is a Hamiltonian cycle of $D$, where $w=v_1$. Then
$P'=v''_1v'_2v''_2\cdots v'_nv''_nv'_1$
is a Hamiltonian path in $D'$. Define $A_2=A(P')$ and $A_1=A(D')\setminus A_2$, and set $D_1=(V(D'),A_1)$ and $D_2=(V(D'),A_2)$. Clearly, $P'$ is an out-branching rooted at $v''_1$ in $D_2$, and every vertex $x\in V(D')$ satisfies $d_{D_1}(x)=1$. Consequently, $(A_1,A_2)$ is the desired arc-decomposition of $D'$.

Conversely, let $(A_1,A_2)$ be a (perfect matching, having $B^+$)-arc-decomposition of $D'$. Then $|A_1|=n$ and
\[
|A_2|=(3n-1)-|A_1|=2n-1.
\]
Since $A_1$ is a perfect matching of $D'$, we obtain $d_{D_2}(w')=d_{D_2}(w'')=1$, and $d_{D_2}(v')=d_{D_2}(v'')=2$ for all $v\in V(D)\setminus\{w\}$. Now $D_2$ contains a spanning out-branching, which already has $2n-1$ arcs. Since $|A_2|=2n-1$, $D_2$ itself is that out-branching. Moreover, $w''$ has in-degree zero in $D'$, so it is the root, while $w'$ has out-degree zero in $D'$, so it is a leaf. Every other vertex has total degree two in $D_2$ and in-degree one in the out-branching, and therefore out-degree one. Hence $D_2$ is a Hamiltonian directed path from $w''$ to $w'$. Again let $v_1=w$; by the vertex-splitting construction, this Hamiltonian path can be written as
$P'=v''_1v'_2v''_2\cdots v'_nv''_nv'_1$.
This implies that all arcs $v_1v_2,v_2v_3,\ldots,v_nv_1$ exist in $D$. Therefore, $C=v_1v_2\cdots v_nv_1$ is a Hamiltonian cycle of $D$.
\end{proof}

As in the previous proof, we can construct the digraph $D'$ from $D$ in linear time. The problem belongs to NP, since both the perfect-matching condition and the existence of an out-branching can be checked in polynomial time. Hence, by Theorem~2.1, the (perfect matching, having $B^+$)-arc-decomposition problem is NP-complete. \hfill$\square$

\section{Arc-decompositions of strong digraphs}

The four reductions below follow the same general approach. Starting from an instance of the corresponding unrestricted problem, we add vertices and arcs to obtain a strong digraph. The added structure ensures that a required arc-decomposition of the original digraph can be extended to the new digraph, and that any such partition of the new digraph yields one of the original.

\subsection{(perfect matching, having no odd cycle)-arc-decompositions}

\begin{proof}
We reduce from the unrestricted (perfect matching, having no odd cycle)-arc-decomposition problem. 
Let $D=(V,A)$ be an instance, and let $s,t,u_1,u_2$ be four distinct new vertices. We construct a digraph $D'=(V',A')$, where $V'=V\cup\{s,t,u_1,u_2\}$ and $A'=A\cup\{sv|v\in V\}\cup\{vt|v\in V\}\cup\{ts,su_1,u_1t,u_1u_2,u_2u_1\}$. The construction is illustrated in Figure~\ref{fig:5}. A routine check confirms that $D'$ is strong.

\begin{figure}[h]
    \centering
    \includegraphics[width=0.32\hsize]{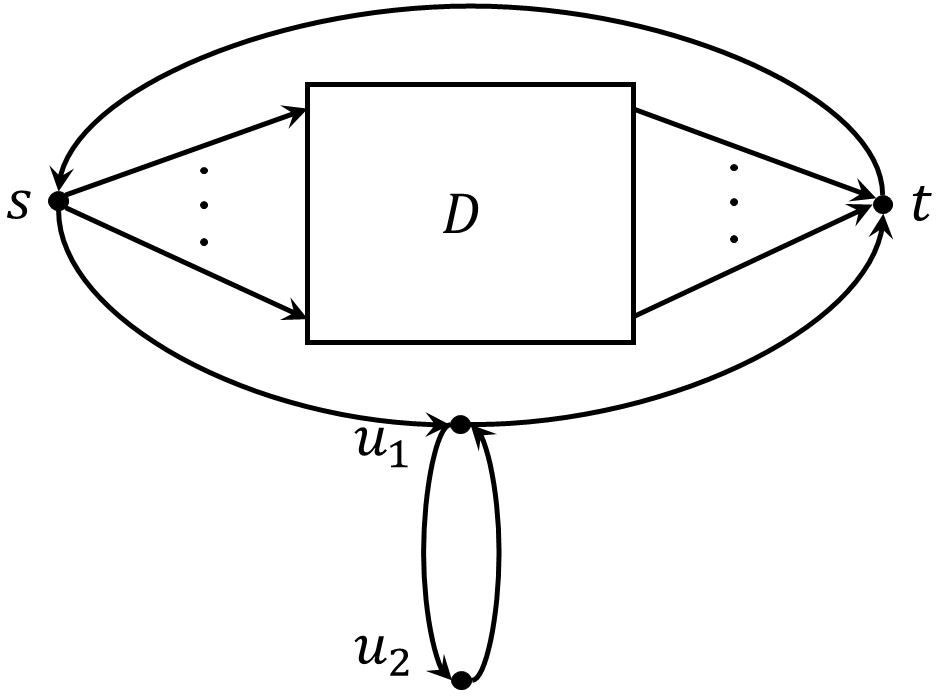}
    \caption{The digraph $D'$ constructed in Section~5.1.}
    \label{fig:5}
\end{figure}

We show that $D$ admits a $(\text{perfect matching}, \text{having no odd cycle})$-arc-decomposition if and only if $D'$ does.

Suppose first that $D$ admits an arc-decomposition $(A_1,A_2)$ such that $A_1$ forms a perfect matching and $A_2$ contains no odd cycle. We define $A_1'=A_1\cup\{ts,u_1u_2\}$ and $A_2'=A'\setminus A_1'$. Clearly, $A_1'$ forms a perfect matching of $D'$. Let $D_2'=(V',A_2')$. Since $d^-_{D_2'}(s)=d^+_{D_2'}(t)=0$, neither $s$ nor $t$ lies on a cycle of $D_2'$. Moreover, $u_1u_2\notin A_2'$, so $u_2$ has in-degree zero in $D_2'$ and cannot lie on a cycle. The only arc entering $u_1$ in $D_2'$ is $u_2u_1$, so $u_1$ cannot lie on a cycle either. Thus every cycle of $D_2'$ is contained in $D_2'[V]=(V,A_2)$. Since $(V,A_2)$ contains no odd cycle, neither does $D_2'$. Therefore, $(A_1', A_2')$ is a (perfect matching, having no odd cycle)-arc-decomposition of $D'$.

Conversely, suppose $D'$ admits a (perfect matching, having no odd cycle)-arc-decomposition $(A_1',A_2')$, and let $D_2'=(V',A_2')$. Since $u_1$ is the only neighbor of $u_2$, we have $u_1u_2 \in A_1'$ or $u_2u_1 \in A_1'$, and consequently $su_1, u_1t \in A_2'$. If $ts\in A_2'$, then $D_2'$ would contain the directed $3$-cycle $su_1ts$, a contradiction. Hence $ts \in A_1'$, and therefore $A_1' \cap A$ forms a perfect matching of $D$. Since $(V,A_2'\cap A)$ is a subdigraph of $D_2'$, it also contains no odd cycle. Thus $(A_1' \cap A,A_2' \cap A)$ is a (perfect matching, having no odd cycle)-arc-decomposition of $D$.

As we can construct $D'$ from $D$ in linear time, the problem is NP-hard by Theorem~\ref{thm:11}. The problem also belongs to NP, since a proposed arc-decomposition can be checked in polynomial time. Hence the (perfect matching, having no odd cycle)-arc-decomposition problem is NP-complete even when restricted to strong digraphs.
\end{proof}

\subsection{(perfect matching, having $B^+$)-arc-decompositions}

The digraph $D'$ constructed in the proof of Theorem~\ref{thm:13} is not strong, since $w'$ has out-degree zero. We therefore give a separate reduction for the strong case.

\begin{proof}
We reduce from the unrestricted (perfect matching, having $B^+$)-arc-decomposition problem. 

Let $D=(V,A)$ be an instance with $V=\{v_1,\dots,v_n\}$, and let $s,t,a_1,\dots,a_n,b_1,\dots,b_n$ be distinct new vertices. We construct a digraph $D'=(V',A')$, where $V'=V\cup\{s,t,a_1,\dots,a_n,b_1,\dots,b_n\}$ and $A'=A\cup\{st,ts\}\cup\{v_it|i\in [n]\}\cup\{ta_i|i\in [n]\}\cup\{a_ib_i|i\in [n]\}\cup\{b_iv_i|i\in [n]\}\cup\{v_ib_i|i\in [n]\}$. See Figure~\ref{fig:6} for an illustration of the construction. A routine check confirms that $D'$ is strong.

\begin{figure}[h]
    \centering
    \includegraphics[width=0.35\hsize]{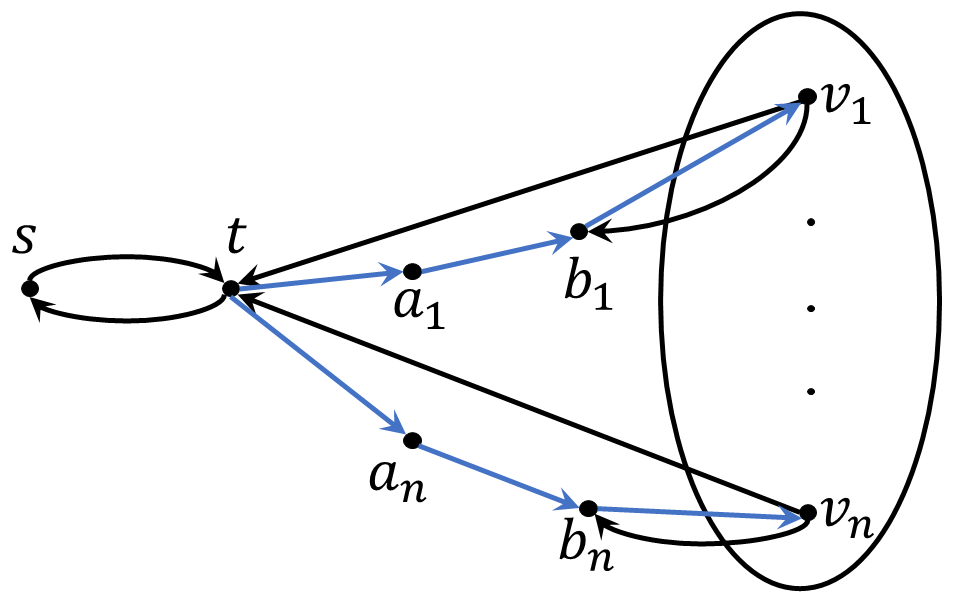}
    \caption{The digraph $D'$ constructed in Section~5.2.}
    \label{fig:6}
\end{figure}

We show that $D$ admits a $(\text{perfect matching}, \text{having}$ $B^+)$-arc-decomposition if and only if $D'$ does.

Suppose first that $D$ admits an arc-decomposition $(A_1,A_2)$ such that $A_1$ forms a perfect matching and $A_2$ contains an out-branching $B^{+}$. Let $A_1'=A_1\cup\{st\}\cup\{a_ib_i|i\in [n]\}$ and $A_2'=A'\setminus A_1'$. Clearly, $A_1'$ forms a perfect matching of $D'$. Relabelling the vertices if necessary, we may assume that  $v_1$ is the root of $B^+$. Adding the arcs in $\{v_1t,ts\}\cup\{ta_i\mid i\in[n]\}\cup\{v_ib_i\mid i\in[n]\}$ extends $B^+$ to an out-branching of $(V',A_2')$. Thus $D'$ admits a (perfect matching, having $B^+$)-arc-decomposition.

Conversely, suppose that $D'$ admits a (perfect matching, having $B^+$)-arc-decomposition $(A_1',A_2')$. Let $D_2'=(V',A_2')$, and let $B^+$ be an out-branching of $D_2'$. Since $s$ has $t$ as its unique neighbour in $D'$, exactly one of $st$ and $ts$ belongs to $A_1'$. In particular, $t$ is already matched to $s$, so $ta_i\in A_2'$ for every $i\in[n]$. Since each $a_i$ has no neighbour other than $t$ and $b_i$, we have $a_ib_i\in A_1'$ for every $i\in[n]$. We claim that $st\in A_1'$. Suppose otherwise. Then $ts\in A_1'$ and $st\in A_2'$. Since $ts$ is the only arc entering $s$, the vertex $s$ has in-degree zero in $D_2'$ and must therefore be the root of $B^{+}$. There is no arc in $D_2'$ from $\{s,t,a_1,\dots,a_n\}$ to $V\cup\{b_1,\dots,b_n\}$. Thus no vertex of $V$ is reachable from $s$, a contradiction. Hence $st\in A_1'$ and $ts\in A_2'$. 
The vertices $s,t,a_1,a_2,\dots,a_n,b_1,b_2,\dots,b_n$ are now all matched among themselves. Consequently, $A_1'\cap A$ is a perfect matching of $D$. 

It remains to show that $D_2=(V,A_2'\cap A)$ contains an out-branching. 
The root of $B^{+}$ cannot be $s$ or any $a_i$, since these vertices have out-degree zero in $D_2'$. It cannot be $t$ either, since from $t$ one can reach only $s$ and the vertices $a_i$, and hence no vertex of $V$. Thus the root is either $v_k$ or $b_k$ for some $k\in[n]$. In the latter case, every path from the root to a vertex of $V$ starts with $b_kv_k$, the unique arc leaving $b_k$. We next observe that, after reaching $v_k$, a directed path in $B^{+}$ that ends in $V$ cannot leave $V$. Indeed, if it uses an arc $v_jt$, then from $t$ it can proceed only to $s$ or to some $a_i$, from which no vertex of $V$ is reachable. If it uses an arc $v_jb_j$, then the only arc leaving $b_j$ is $b_jv_j$, which would repeat $v_j$ and therefore cannot occur on a directed path. Thus, for every $i\in[n]$, the unique path from $v_k$ to $v_i$ in $B^+$ lies entirely in $V$. Hence $B^{+}[V]$ is an out-branching of $D_2$ rooted at $v_k$. Therefore, $(A_1'\cap A,A_2'\cap A)$ is a (perfect matching, having $B^+$)-arc-decomposition of $D$.

As we can construct $D'$ from $D$ in linear time, the problem is NP-hard by Theorem~\ref{thm:13}. The problem also belongs to NP, since a proposed arc-decomposition can be checked in polynomial time. Hence the (perfect matching, having $B^+$)-arc-decomposition problem is NP-complete even when restricted to strong digraphs.
\end{proof}

\subsection{(cycle factor, having no odd cycle)-arc-decompositions}

\begin{proof}
We reduce from the unrestricted (cycle factor, having no odd cycle)-arc-decomposition problem. Let $D=(V,A)$ be an instance, and let $s,t,r$ be three distinct new vertices. We construct a digraph $D'=(V',A')$, where $V'=V\cup\{s,t,r\}$ and $A'=A\cup\{sr,rt,ts\}\cup\{sv|v\in V\}\cup\{vt|v\in V\}$. The construction is illustrated in Figure~\ref{fig:7}. A routine check confirms that $D'$ is strong.

\begin{figure}[h]
    \centering
    \includegraphics[width=0.32\hsize]{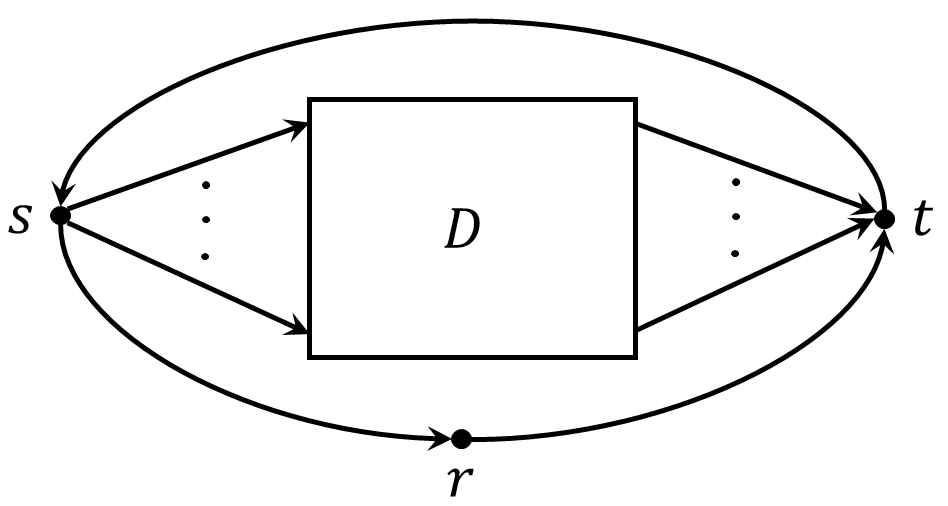}
    \caption{The digraph $D'$ constructed in Section~5.3.}
    \label{fig:7}
\end{figure}

We show that $D$ admits a $(\text{cycle factor}, \text{having no odd cycle})$-arc-decomposition if and only if $D'$ does.

Suppose first that $(A_1,A_2)$ is a (cycle factor, having no odd cycle)-arc-decomposition of $D$. Let $A_1'=A_1\cup\{sr,rt,ts\}$ and $A_2'=A'\setminus A_1'$. Since $D_1'=(V',A_1')$ consists of the cycle factor $(V,A_1)$ together with the directed $3$-cycle $srts$, $D_1'$ is a cycle factor. Let $D_2'=(V',A_2')$. Since $d^-_{D_2'}(s)=d^+_{D_2'}(t)=d_{D_2'}(r)=0$, no cycle of $D_2'$ contains $s$, $t$ or $r$. Therefore, every cycle of $D_2'$ is contained in $D_2'[V]=(V,A_2)$. Since $(V,A_2)$ contains no odd cycle,
neither does $D_2'$. 

Conversely, suppose that $(A_1',A_2')$ is a (cycle factor, having no odd cycle)-arc-decomposition of $D'$. Since $sr$ and $rt$ are the unique arcs entering and leaving $r$, respectively, both belong to $A_1'$. Moreover, $ts$ is the unique arc leaving $t$, and therefore $ts\in A_1'$. Thus the directed $3$-cycle $srts$ is a component of the cycle factor $(V',A_1')$. Removing this cycle leaves a cycle factor of $D$, namely $(V,A_1'\cap A)$. Since $(V,A_2'\cap A)$ is a subdigraph of $(V',A_2')$, it also contains no odd cycle. Therefore, $(A_1'\cap A,A_2'\cap A)$ is a (cycle factor, having no odd cycle)-arc-decomposition of $D$.

As we can construct $D'$ from $D$ in linear time, the problem is NP-hard by Theorem~\ref{thm:14}. The problem also belongs to NP, since a proposed arc-decomposition can be checked in polynomial time. Hence the (cycle factor, having no odd cycle)-arc-decomposition problem is NP-complete even when restricted to strong digraphs.
\end{proof}

\subsection{Two arc-disjoint perfect matchings}

\begin{proof}
We reduce from the unrestricted (having a perfect matching, having a perfect matching)-arc-decomposition problem. Let $D=(V,A)$ be an instance, and let $s_1,s_2,t_1,t_2,r_1,r_2$ be six distinct new vertices. We construct a digraph $D'=(V',A')$, where $V'=V\cup\{s_1,s_2,t_1,t_2,r_1,r_2\}$ and $A'=A\cup\{s_1s_2,s_2s_1,t_1t_2,t_2t_1,r_1r_2,r_2r_1,t_2r_1,r_1s_2\}\cup\{s_1v|v\in V\}\cup\{vt_1|v\in V\}$. Figure~\ref{fig:8} shows the resulting digraph $D'$. A routine check confirms that $D'$ is strong.

\begin{figure}[h]
    \centering
    \includegraphics[width=0.32\hsize]{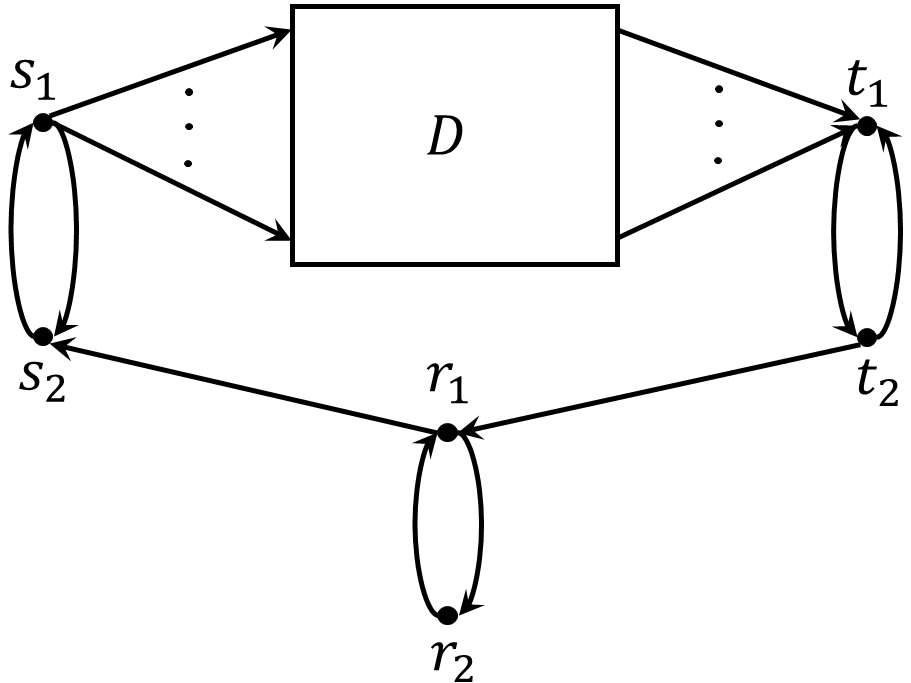}
    \caption{The digraph $D'$ constructed in Section~5.4.}
    \label{fig:8}
\end{figure}

We show that $D$ admits a $(\text{having a perfect matching}, \text{having a perfect matching})$-arc-decomposition if and only if $D'$ does.

Suppose first that $(A_1,A_2)$ is a (having a perfect matching, having a perfect matching)-arc-decomposition of $D$, and let $M_i$ be a perfect matching of $D[A_i]$ for $i=1,2$. Let $A_1'=A_1\cup\{s_1s_2,t_1t_2,r_1r_2\}$ and $A_2'=A'\setminus A_1'$. Note that $D'[A_1']$ contains the perfect matching $M_1\cup\{s_1s_2,t_1t_2,r_1r_2\}$ while $D'[A_2']$ contains the perfect matching $M_2\cup\{s_2s_1,t_2t_1,r_2r_1\}$. Hence $D'$ admits the required arc-decomposition.

Conversely, suppose that $(A_1',A_2')$ is a (having a perfect matching, having a perfect matching)-arc-decomposition of $D'$, and let $M_i'$ be a perfect matching of $D'[A_i']$ for $i=1,2$. Observe that the only neighbour of $r_2$ in the underlying graph of $D'$ is $r_1$. Consequently, neither $t_2r_1$ nor $r_1s_2$ can belong to any perfect matching of $D'$. Since $s_2$ cannot be matched with $r_1$, it must be matched with $s_1$. Similarly, since $t_2$ cannot be matched with $r_1$, it must be matched with $t_1$. Thus every perfect matching of $D'$ matches $s_1$ with $s_2$ and $t_1$ with $t_2$. Hence no arc between $\{s_1,t_1\}$ and $V$ belongs to any perfect matching of $D'$. It follows that $M_1'\cap A$ and $M_2'\cap A$ are two arc-disjoint perfect matchings of $D$. Let $A_1=M_1'\cap A$ and $A_2=A\setminus A_1$. Since $M_2'\cap A\subseteq A_2$, $(A_1,A_2)$ is the required arc-decomposition of $D$.

The construction of $D'$ takes polynomial time, so the above reduction establishes NP-hardness. The problem also belongs to NP, since a proposed arc-decomposition can be checked in polynomial time. Hence the (having a perfect matching, having a perfect matching)-arc-decomposition problem is NP-complete even when restricted to strong digraphs.
\end{proof}

\section{Proof of Theorem~\ref{thm:16}}

\begin{proof}
Let $D$ be a strong digraph with minimum semi-degree at least $3$. By Theorem~\ref{thm:2cycles}, $D$ contains two vertex-disjoint cycles $C_1,C_2$. Let $H=D-A(C_1)-A(C_2)$. Since $C_1$ and $C_2$ are vertex-disjoint, we have $\delta^0(H) \ge 2$. 

If $H$ is strong, the proof is complete. Hence, we may assume that $H$ is not strong. Then $H$ has at least two strong components. We may take a source strong component $S$ and a sink strong component $T$ such that $S\neq T$. Since $T$ is a sink strong component, $d^+_T(v)=d^+_H(v)\ge 2$ for every $v\in V(T)$, and hence $\delta^+(T)\ge 2$. Since $T$ is strong, by Proposition~\ref{thm:cycleandstrong}, there exists a cycle $C_T$ such that $T-A(C_T)$ is strong. 
Since $S$ is a source strong component, $d^-_S(v)=d^-_H(v)\ge 2$ for every $v\in V(S)$, and hence $\delta^-(S)\ge 2$. Let $S'$ be the digraph obtained by reversing all arcs of $S$. Then $S'$ is strong and $\delta^+(S')\ge 2$. By Proposition~\ref{thm:cycleandstrong}, there exists a cycle $C_{S'}$ such that $S'-A(C_{S'})$ is strong. This implies that there exists a cycle $C_S$ in $S$ such that $S-A(C_S)$ is strong. 
Since $S\ne T$, $C_T$ and $C_S$ are vertex-disjoint. 
Since $T-A(C_T)$ is a strong spanning subdigraph of $D[V(T)]-A(C_T)$, the latter is strong. Similarly, $D[V(S)]-A(C_S)$ is strong. 

Let $u,v\in V(D)$ be distinct, and let $P$ be a directed path from $u$ to $v$ in $D$. If $P$ does not use any arc of $A(C_T)\cup A(C_S)$, then it is also a directed path in $D-A(C_T)-A(C_S)$. If $P$ uses an arc $e=(a,b)$ in $A(C_T)$, then since $D[V(T)]-A(C_T)$ is strongly connected, we may replace the arc by a directed path from $a$ to $b$ in $D[V(T)]-A(C_T)$. Similarly, we may replace the arcs in $A(P) \cap A(C_S)$. Since $V(S)$ and $V(T)$ are disjoint, all replacement paths avoid $A(C_S)\cup A(C_T)$. After replacing all such arcs, we obtain a directed walk from $u$ to $v$, which contains a directed path from $u$ to $v$ in $D-A(C_T)-A(C_S)$. Since $u$ and $v$ were arbitrary, $D-A(C_T)-A(C_S)$ is strong.
\end{proof}

\section*{Acknowledgements}

This work was supported by the National Natural Science Foundation of China (Grant No.~12571373) and the Natural Science Foundation of Shandong Province (Grant No.~ZR2025MS05).

\bibliographystyle{plain} 
\bibliography{ref}

\clearpage
\section*{Appendix: Proof of Theorem~\ref{thm:14}}

We use the same construction as in \cite[Section~4.3.3]{BJBGPA2022}, but provide the verification for the (cycle factor, having no odd cycle)-arc-decomposition problem.

\begin{proof}[Proof of Theorem \ref{thm:14}]

We show how to reduce the $3$-SAT problem to the (cycle factor, having no odd cycle)-arc-decomposition problem. 

Let $F$ be a $3$-SAT instance with variables $x_1,\dots,x_n$ and clauses $C_1,\dots,C_m$. Note that we require every clause to consist of three pairwise distinct and non-complementary literals, and every variable to occur in at least one clause. The prescribed order of clauses naturally defines an order for all occurrences of each variable $x$ and its negation $\overline{x}$ within $F$. For nonnegative integers $p,q$ with $p+q>0$, define the variable gadget $W[u,v;p,q]$ with vertex set
\[
V(W)=\{u,v,y_1,\ldots,y_{3p},z_1,\ldots,\\z_{3q}\}
\]
and arc set
\[
A(W)=A(P_y)\cup A(P_z)\cup\{y_{3i}y_{3i-2}|i\in[p]\}\cup\{z_{3j}z_{3j-2}|j\in[q]\}\cup\{vu\}
\]
where $P_y=uy_1\ldots y_{3p}v$ and $P_z=uz_1\ldots z_{3q}v$. If $p=0$ (respectively, $q=0$), the corresponding path is understood to be the single arc $uv$. The variable gadget $W[u,v;2,1]$ is illustrated in Figure~\ref{fig:9}.

\begin{figure}[h]
    \centering
    \includegraphics[width=0.4\hsize]{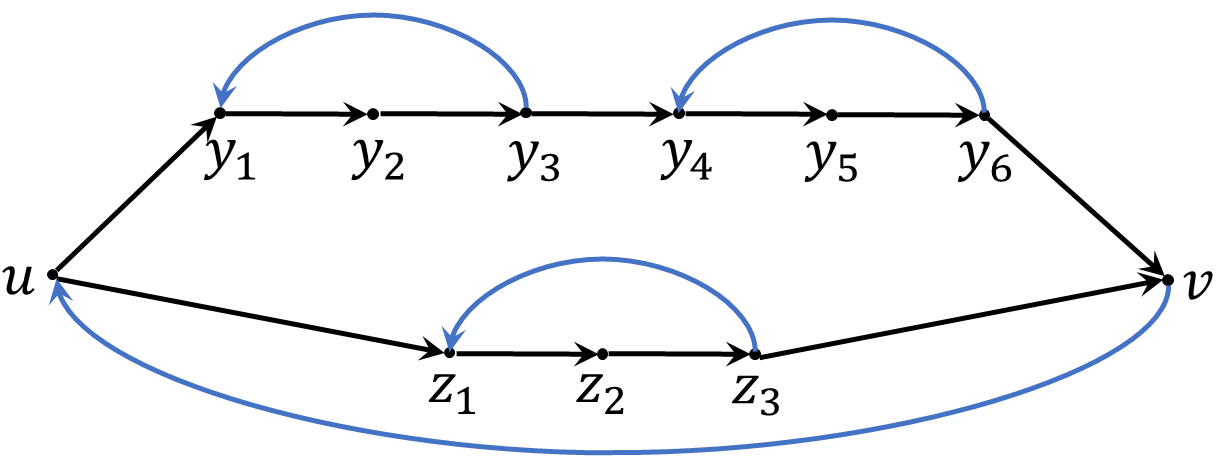}
    \caption{$W[u,v;2,1]$}
    \label{fig:9}
\end{figure}

\begin{property}
Suppose $(A_1,A_2)$ is an arc-decomposition of $W[u,v;p,q]$ such that $A_1$ forms a cycle factor and $A_2$ contains no odd cycles. Then the following properties hold:
\begin{enumerate}[label=\arabic*.]
\item $vu \in A_1$;
\item Exactly one of the first arcs of $P_y$ and $P_z$ belongs to $A_1$;
\item The first and last arcs of $P_y$, together with all arcs $y_{3i}y_{3i+1}$ for $1\le i<p$, belong to the same part. The analogous statement holds for $P_z$.

\end{enumerate}
\end{property}

\begin{proof}
As $vu$ is the unique outgoing arc of vertex $v$, the cycle factor of $W[u,v;p,q]$ must contain this arc. 
As the first arcs of $P_y$ and $P_z$ are the only two outgoing arcs of $u$, 

any cycle factor must contain exactly one of them. 
As each vertex $y_{3i-1}$, $i\in[p]$, and each vertex $z_{3j-1}$, $j\in[q]$, has a unique in-arc and a unique out-arc, arcs $y_{3i-2}y_{3i-1},y_{3i-1}y_{3i},z_{3j-2}z_{3j-1},z_{3j-1}z_{3j}$ all belong to $A_1$. And among the arcs entering $y_{3i-2}$ and $z_{3j-2}$ (respectively, leaving $y_{3i}$ and $z_{3j}$) exactly one should be in $A_1$. Therefore, the first and last arcs of $P_y$, together with all arcs $y_{3i}y_{3i+1}$ for $1\le i<p$, belong to the same part. The analogous statement holds for $P_z$.
\end{proof}

Property $2$ remains valid when $W[u,v;p,q]$ is embedded in the overall digraph $D_F$. In fact, every arc leaving $W[u,v;p,q]$ originates from vertices $y_{3i+1}$ and $z_{3i+1}$, and every arc entering this gadget ends at vertices $y_{3i}$ and $z_{3i}$. All these arcs must belong to set $A_2$.

We next construct the clause gadget $R[r,s,t;e_1,e_2,e_3]$, where $r,s,t$ are vertices and $e_1=a_1b_1$, $e_2=a_2b_2$, $e_3=a_3b_3$ are arcs. The gadget $R$ has $9$ vertices and the arc set consists of a directed $7$-cycle $a_1b_1a_2b_2a_3b_3r$ and a directed $3$-cycle $rst$. The clause gadget is illustrated in Figure~\ref{fig:10}.

\begin{figure}[h]
    \centering
    \includegraphics[width=0.26\hsize]{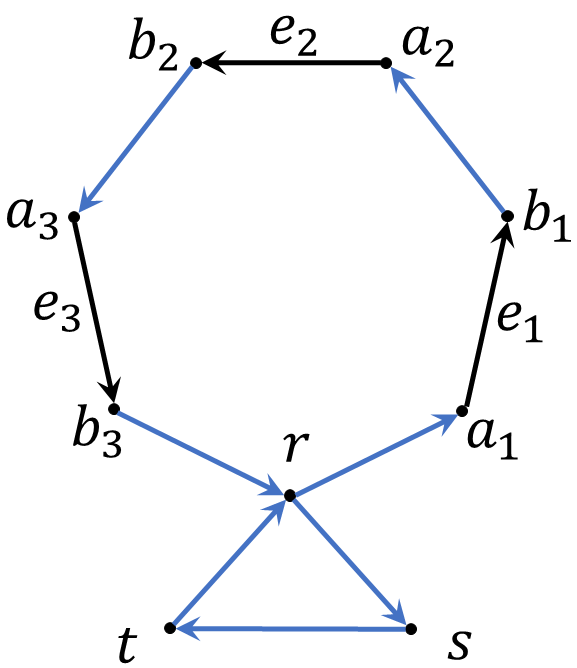}
    \caption{$R[r,s,t;e_1,e_2,e_3]$}
    \label{fig:10}
\end{figure}

We construct the digraph $D_F$ according to the following rules:
\begin{itemize}
    \item[(1)] For each variable $x_i$, we assign a distinct copy $W[u_i,v_i;2p_i,2q_i]$ of the variable gadget, where $u_i$ and $v_i$ are vertices. Here $p_i$ (respectively, $q_i$) denotes the number of clauses in which $x_i$ (respectively, $\overline{x_i}$) occurs as a literal.
    \item[(2)] For each clause $C_j$, we take one copy $R[r_j,s_j,t_j;e_{j,1},e_{j,2},e_{j,3}]$ of the clause gadget where the three vertices $r_j,s_j,t_j$ are private to this copy and the arcs $e_{j,1},e_{j,2},e_{j,3}$ are identified with some arcs from the variable gadgets as follows. Suppose $C_j$ consists of literals from variables $x_i,x_k,x_l$, where each literal may be either a variable or its negation. If $x_i$ appears unnegated in $C_j$ and this is the $r$-th occurrence of $x_i$, then we identify $e_{j,1}$ with the arc $y_{i,6r-3}y_{i,6r-2}$. If $C_j$ contains $\overline{x_i}$ and this is the $s$-th occurrence of $\overline{x_i}$, then we identify $e_{j,1}$ with the arc $z_{i,6s-3}z_{i,6s-2}$. The arcs $e_{j,2}$ and $e_{j,3}$ are defined analogously with arcs of some variable gadgets.
\end{itemize}
\begin{claim}
    $F$ is satisfiable if and only if $D_F$ admits a (cycle factor, having no odd cycle)-arc-decomposition.
\end{claim}
\begin{proof}
Suppose $F$ is satisfiable, and let $\Phi$ be a satisfying truth assignment for $F$. For each variable $x_i$, if $x_i$ evaluates to true under $\Phi$, we include in $A_1$ all arcs of the cycle $u_iy_{i,1}\cdots y_{i,6p_i}v_iu_i$ and all arcs of the cycle $z_{i,3j-2}z_{i,3j-1}z_{i,3j}z_{i,3j-2}$ for every $j\in[2q_i]$. If $x_i$ evaluates to false under $\Phi$, we instead place into $A_1$ the arcs of the cycle $u_iz_{i,1}\cdots z_{i,6q_i}v_iu_i$ and all arcs of the cycle $y_{i,3j-2}y_{i,3j-1}y_{i,3j}y_{i,3j-2}$ for every $j\in[2p_i]$. We further add to $A_1$ the arcs of the $3$-cycle $r_ks_kt_k$ for all $k\in[m]$, and define $A_2$ to be the set of all remaining arcs. One can verify that every arc associated with a true literal lies in $A_1$.

Clearly, every vertex of $D_F$ has exactly one incoming arc and one outgoing arc within $A_1$. Consequently, the digraph $(V(D_F),A_1)$ forms a cycle factor.

If two arcs $e,e'\in A_2$ originate from distinct clause gadgets, they lie in different connected components of $(V(D_F),A_2)$. Furthermore, any arc from a variable gadget that belongs to $A_2$ is either contained in a clause gadget or adjacent to at most one other arc in $A_2$. To show that $(V(D_F),A_2)$ contains no odd cycles, it suffices to prove that the $7$-cycle inside each clause gadget cannot have all its arcs in $A_2$. This fact holds since each clause contains at least one true literal under the assignment $\Phi$.

Conversely, suppose $D_F$ admits an arc-decomposition $(A_1,A_2)$ such that $A_1$ forms a cycle factor and $A_2$ contains no odd cycles. We define a truth assignment $\Phi$ as follows: for each $i\in[n]$, set $x_i$ to true if the first arc of $P_y$ in the variable gadget associated with $x_i$ belongs to $A_1$, and set $x_i$ to false otherwise.

Consider an arbitrary clause $C_j$ of $F$. As noted earlier, all arcs of the $7$-cycle in the corresponding clause gadget, except $e_{j,1},e_{j,2},e_{j,3}$, lie in $A_2$. Since the subdigraph induced by $A_2$ contains no odd cycles, at least one arc among $e_{j,1},e_{j,2},e_{j,3}$ must belong to $A_1$. Without loss of generality, let $e_{j,1}\in A_1$.

If the associated variable $x_i$ appears as a positive literal in $C_j$, then $e_{j,1}$ belongs to the path $P_y = u_iy_{i,1}\cdots y_{i,6p_i}v_i$. This implies that the first arc of $P_y$ belongs to $A_1$, so $x_i$ is true under $\Phi$. Similarly, if $x_i$ appears as a negative literal in $C_j$, then $e_{j,1}$ lies in the path $P_z = u_iz_{i,1}\cdots z_{i,6q_i}v_i$. It follows that the first arc of $P_z$ belongs to $A_1$, so $x_i$ is false under $\Phi$. Since $C_j$ is chosen arbitrarily, every clause of $F$ is satisfied by the assignment $\Phi$.
\end{proof}
As we can construct $D_F$ in linear time, the problem is NP-hard. It also belongs to NP, since a proposed arc-decomposition can be checked in polynomial time: one verifies that $A_1$ is a cycle factor and that $(V(D_F), A_2)$ contains no odd cycles. Hence, the (cycle factor, having no odd cycle)-arc-decomposition problem is NP-complete.
\end{proof}

\end{document}